\documentclass[reqno,12pt]{amsart}
\usepackage{amsfonts}
\usepackage{bbm}
\usepackage{} 
\numberwithin{equation}{section}
\usepackage{indentfirst}
\usepackage{color}
\usepackage{amssymb}
\usepackage{mathrsfs}
\usepackage[shortcuts]{extdash}
\usepackage{xy}
\xyoption{all}

\theoremstyle{plain} 
\newtheorem{theorem}{\bf Theorem}[section]
\newtheorem{lemma}[theorem]{\bf Lemma}
\newtheorem{corollary}[theorem]{\bf Corollary}
\newtheorem{proposition}[theorem]{\bf Proposition}

\theoremstyle{definition} 
\newtheorem{definition}[theorem]{\bf Definition}
\newtheorem{condition}[theorem]{\bf Condition}
\newtheorem{remark}[theorem]{\bf Remark}
\newtheorem{example}[theorem]{\bf Example}

\newcommand{\bt}{\begin{theorem}}
\newcommand{\et}{\end{theorem}}
\newcommand{\bl}{\begin{lemma}}
\newcommand{\el}{\end{lemma}}
\newcommand{\bd}{\begin{definition}}
\newcommand{\ed}{\end{definition}}
\newcommand{\bc}{\begin{corollary}}
\newcommand{\ec}{\end{corollary}}
\newcommand{\bp}{\begin{proof}}
\newcommand{\ep}{\end{proof}}
\newcommand{\bx}{\begin{example}}
\newcommand{\ex}{\end{example}}
\newcommand{\br}{\begin{remark}}
\newcommand{\er}{\end{remark}}
\newcommand{\be}{\begin{equation}}
\newcommand{\ee}{\end{equation}}
\newcommand{\ba}{\begin{align}}
\newcommand{\ea}{\end{align}}
\newcommand{\bn}{\begin{enumerate}}
\newcommand{\en}{\end{enumerate}}
\newcommand{\bcs}{\begin{cases}}
\newcommand{\ecs}{\end{cases}}

\makeatletter
\renewcommand{\section}{\@startsection{section}{1}{0mm}
  {-\baselineskip}{0.5\baselineskip}{\bf\leftline}}
\makeatother

\begin{document}

\title[Wakamatsu-tilting subcategories]{Wakamatsu-tilting subcategories in extriangulated categories} 

\author[Z. Zhu, J. Wei]{Zhiwei Zhu, Jiaqun Wei$^\ast$}

\address{Institute of Mathematics, School of Mathematical Sciences, Nanjing Normal University, Nanjing, 210023, P.~R.~China}
\email{1985933219@qq.com (Zhu)}

\address{Department of Mathematics, Zhejiang Normal University, Jinhua, 321004, Zhejiang, P.~R.~China}
\email{weijiaqun5479@zjnu.edu.cn (Wei)}


\thanks{$*$~Corresponding author.}


\begin{abstract}
Let $\mathscr{C}$ be an extriangulated category with enough projectives and injectives. We give the definitions of Wakamatsu-tilting subcategories and Wakamatsu-cotilting subcategories of $\mathscr{C}$ and show that they coincide with each other. Moreover, the definitions of $\infty$-tilting subcategories and $\infty$-cotilting subcategories given by Zhang, Wei and Wang also coincide with them. As a result, Wakamatsu-tilting subcategories success all properties of $\infty$-tilting subcategories and $\infty$-cotilting subcategories. On the other hand, we glue the Wakamatsu-tilting subcategories in a special recollement and show that the converse of the gluing holds under certain conditions.
\end{abstract}

\maketitle

\section{Introduction}
Tilting theory, as a generation of Morita equivalences, plays an important role in the representation theory of algebra. It originated with the study of reflection functors in \cite{2,1}. Then Brenner and Butler described the first set of axioms for a titling module in \cite{3}. A further generalization of tilting modules to modules of possibly infinite projective dimension was made by Wakamatsu in \cite{9}. Such a generalized tilting module is now known as Wakamatsu tilting module. The present paper is devoted to studying various aspects of these modules in \cite{10,11}. Nowadays, tilting theory has been generalized in many directions. Krause \cite{5} defined tilting objects in exact categories and Sauter \cite{6} defined tilting subcategories in exact categories. Wakamatsu tilting subcategory is a certain categorical analogue of Wakamatsu tilting module made in \cite{12}. Recently, Zhu and Zhuang \cite{7} generalized the definition of tilting subcategories in extriangulated categories. In section 3, we will introduce Wakamatsu-tilting subcategories in extriangulated categories and obtain some results.

Extriangulated categories were introduced by Nakaoka and Palu \cite{8}, which share some properties of triangulated categories and exact categories. There are some examples of extriangulated categories such as exact categories and extension-closed subcategories of triangulated categories, while some extriangulated categories may be neither exact nor triangulated categories. However, some results of triangulated categories and exact categories can be generalized to extriangulated categories.

The concept of a recollement in triangulated categories was first introduced by Beilinson, Bernstein and Deligne in \cite{13}. Recollements in abelian categories appeared in the construction of perverse sheaves by MacPherson and Vilonen \cite{14}. Recently, Wang, Wei, and Zhang introduced the concept of recollements of extriangulated categories in \cite{16}, which provides a unified generalization of recollements in both abelian and triangulated categories. Furthermore, researchers have delved into gluing techniques with respect to recollements of extriangulated categories. In section 4, we will glue together Wakamatsu-tilting subcategories in extriangulated categories.

The paper is organized as follows. In Section 2, we summarize some definitions and results of extriangulated categories. In section 3, we introduce Wakamatsu-tilting subcategories and Wakamatsu-cotilting subcategories in extriangulated categories. Then we will show that these two definitions and the definitions of $\infty$-tilting subcategories and $\infty$-cotilting subcategories in \cite{15} are equivalent. In section 4, we will glue together Wakamatsu-tilting subcategories in a special recollement of extriangulated categories and show that the converse of the gluing holds under certain conditions.

\section{Preliminaries}
Throughout the article, we assume, unless otherwise stated, that $\mathscr{C}$ denotes an additive category, which is skeletally small and Krull-Schmidt. All subcategories
considered are full and closed under isomorphisms. We denote by $\mathscr{C}(A,B)$ the set of morphisms from $A$ to $B$ in $\mathscr{C}$. The composition of $a\in\mathscr{C}(A,B)$ and $b\in\mathscr{C}(B,C)$ is denoted by $ba$. For a subcategory $\mathscr{A}$ of $\mathscr{C}$, $a\in\mathscr{C}(A,C)$ is a {\em right} $\mathscr{A}$-{\em approximation} for $C\in\mathscr{C}$ if $A\in\mathscr{A}$ and $\mathscr{C}(A',a)$ is surjective for any $A'\in\mathscr{A}$. Dually, we can define {\em left $\mathscr{A}$-approximation}.

\subsection{Extriangulated categories}
Let us recall some notions concerning extriangulated categories from \cite{8}.

Let $\mathbb{E}:\mathscr{C}^{op}\times\mathscr{C}\rightarrow Ab$ be a biadditive functor, where $Ab$ is the category of abelian groups. For any pair of objects $A,C\in\mathscr{C}$, an element $\delta\in\mathbb{E}(C,A)$ is called an
$\mathbb{E}$-{\em extension}. The zero element $0\in\mathbb{E}(C,A)$ is called the {\em split} $\mathbb{E}$-{\em extension}. For any morphism $a\in\mathscr{C}(A,A')$ and $c\in\mathscr{C}(C',C)$, we have the following $\mathbb{E}$-extensions
$$\mathbb{E}(C,a)(\delta)\in\mathbb{E}(C,A'), \,\mathbb{E}(c,A)(\delta)\in\mathbb{E}(C',A),$$
which are denoted by $a_*\delta$ and $c^*\delta$, respectively.
\begin{definition}\label{morphism}
   \cite[Definition 2.3]{8} A morphism $(a,c):\delta\rightarrow \delta'$ of $\mathbb{E}$-extensions $\delta\in\mathbb{E}(C,A)$, $\delta'\in\mathbb{E}(C',A')$ is a pair of morphisms $a\in\mathscr{C}(A,A')$ and $c\in\mathscr{C}(C,C')$ satisfying  $a_*\delta=c^*\delta'$.
\end{definition}
Two sequences of morphisms $A\stackrel{x}{\longrightarrow}B\stackrel{y}{\longrightarrow}C$ and $A\stackrel{x'}{\longrightarrow}B'\stackrel{y'}{\longrightarrow}C$ in $\mathscr{C}$ are said to be {\em equivalent} if there exists an isomorphism $b\in\mathscr{C}(B,B')$ such that the following diagram is commutative.
$$\xymatrix{A \ar[r]^x \ar@{=}[d]& B\ar[r]^y \ar[d]^b_{\simeq}&C\ar@{=}[d]\\
A\ar[r]^{x'}&B'\ar[r]^{y'}&C}$$
We denote the equivalence class of $A\stackrel{x}{\longrightarrow}B\stackrel{y}{\longrightarrow}C$ by $[A\stackrel{x}{\longrightarrow}B\stackrel{y}{\longrightarrow}C]$. In addition, for any $A,C\in\mathscr{C}$, we denote as
$$0=[\xymatrix{A\ar[r]^{\tiny{\setlength{\arraycolsep}{1.2pt}\begin{pmatrix}1\\0\end{pmatrix}}\quad\quad}&A\oplus B\ar[r]^{\quad\tiny{\setlength{\arraycolsep}{1.2pt}\begin{pmatrix}0 & 1\end{pmatrix}}}&C}].
$$
For any two equivalence classes $[A\stackrel{x}{\longrightarrow}B\stackrel{y}{\longrightarrow}C]$ and $[A'\stackrel{x'}{\longrightarrow}B'\stackrel{y'}{\longrightarrow}C']$ we denote as
$$[A\stackrel{x}{\longrightarrow}B\stackrel{y}{\longrightarrow}C]\oplus [A'\stackrel{x'}{\longrightarrow}B'\stackrel{y'}{\longrightarrow}C'] = [A\oplus A'\stackrel{\tiny{\setlength{\arraycolsep}{1.2pt}\begin{pmatrix}
                    x & 0 \\
                    0 & x'
                  \end{pmatrix}}}{\longrightarrow}B\oplus B'\stackrel{\tiny{\setlength{\arraycolsep}{1.2pt}\begin{pmatrix}
                                      y & 0 \\
                                      0 & y'
                                    \end{pmatrix}}}{\longrightarrow}C\oplus C'].$$
\begin{definition}
  \cite[Definition 2.9]{8} Let $\mathfrak{s}$ be a correspondence, which associates an equivalence class $\mathfrak{s}(\delta)=[A\stackrel{x}{\longrightarrow}B\stackrel{y}{\longrightarrow}C]$ to each $\mathbb{E}$-extension $\delta\in\mathbb{E}(C,A)$. This $\mathfrak{s}$ is called a {\em realization} of $\mathbb{E}$ if for any morphism $(a,c):\delta\rightarrow \delta'$ with $\mathfrak{s}(\delta)=[A\stackrel{x}{\longrightarrow}B\stackrel{y}{\longrightarrow}C]$ and $\mathfrak{s}(\delta')=[A'\stackrel{x'}{\longrightarrow}B'\stackrel{y'}{\longrightarrow}C']$, there is a commutative diagram as follows:
  $$\xymatrix{A \ar[r]^x \ar[d]^a& B\ar[r]^y \ar[d]^b&C\ar[d]^c\\
  A'\ar[r]^{x'}&B'\ar[r]^{y'}&C'}
  $$
  A realization $\mathfrak{s}$ of $\mathbb{E}$ is said to be {\em additive} if the following conditions are satisfied:

  \;(a) For any $A,C\in\mathscr{C}$, the split $\mathbb{E}$-extension $0\in\mathbb{E}(C,A)$ satisfies $\mathfrak{s}(0)=0$.

  \;(b) $\mathfrak{s}(\delta\oplus \delta')=\mathfrak{s}(\delta)\oplus\mathfrak{s}(\delta')$ for any pair of $\mathbb{E}$-extensions $\delta$ and $\delta'$.
\end{definition}

\begin{definition}
  \cite[Definition 2.12]{8} We call the triple $\mathscr{C}=(\mathscr{C},\mathbb{E},\mathfrak{s})$ an extriangulated category if it satisfies the following conditions:

  ${\rm(ET1)}$ $\mathbb{E}:\mathscr{C}^{op}\times\mathscr{C}\rightarrow Ab$ is a biadditive functor.

  ${\rm(ET2)}$ $\mathfrak{s}$ is an additive realization of $\mathbb{E}$.

  ${\rm(ET3)}$ Let $\delta\in\mathbb{E}(C,A)$ and $\delta'\in\mathbb{E}(C',A')$ be any pair of $\mathbb{E}$-extensions, with
  $$\mathfrak{s}(\delta)=[A\stackrel{x}{\longrightarrow}B\stackrel{y}{\longrightarrow}C] \;{\rm and}\;
  \mathfrak{s}(\delta')=[A'\stackrel{x'}{\longrightarrow}B'\stackrel{y'}{\longrightarrow}C'].$$

  \quad\quad\quad For any commutative diagram
  $$\xymatrix{A \ar[r]^x \ar[d]^a& B\ar[r]^y \ar[d]^b&C\\
  A'\ar[r]^{x'}&B'\ar[r]^{y'}&C'}
  $$

  \quad\quad\quad in $\mathscr{C}$, there is a morphism $(a,c):\delta\rightarrow \delta'$ satisfying $cy=y'b$.

  ${\rm(ET3)^{op}}$ Dual of ${\rm(ET3)}$

  ${\rm(ET4)}$ Let $\delta\in\mathbb{E}(D,A)$ and $\delta'\in\mathbb{E}(F,B)$ be $\mathbb{E}$-extensions realized by $A\stackrel{f}{\longrightarrow}B\stackrel{f'}{\longrightarrow}D$

  \quad\quad\quad and $B\stackrel{g}{\longrightarrow}C\stackrel{g'}{\longrightarrow}F$, respectively. Then there exist an object $E\in\mathscr{C}$, a

  \quad\quad\quad commutative diagram
  $$\xymatrix{A \ar[r]^f \ar@{=}[d]& B\ar[r]^{f'}\ar[d]^g&D\ar[d]^d\\
  A\ar[r]^{h}&C\ar[r]^{h'}\ar[d]^{g'}&E\ar[d]^e\\&F\ar@{=}[r]&F}
  $$

  \quad\quad\quad in $\mathscr{C}$, and an $\mathbb{E}$-extension $\delta''\in\mathbb{E}(E,A)$ realized by $A\stackrel{h}{\longrightarrow}C\stackrel{h'}{\longrightarrow}E$, which

 \quad\quad\quad satisfy the following compatibilities:

  \quad\quad\quad (i) $D\stackrel{d}{\longrightarrow}E\stackrel{e}{\longrightarrow}F$ realizes $\mathbb{E}(F,f')(\delta')$,

  \quad\quad\quad (ii) $\mathbb{E}(d,A)(\delta'')=\delta$,

  \quad\quad\quad (iii) $\mathbb{E}(E,f)(\delta'')=\mathbb{E}(e,B)(\delta')$.

  ${\rm(ET4)^{op}}$ Dual of ${\rm(ET4)}$
\end{definition}

\begin{remark}
  (a) A sequence $A\stackrel{x}{\longrightarrow}B\stackrel{y}{\longrightarrow}C$ is called a {\em conflation} if it realizes some $\mathbb{E}$-extension $\delta\in\mathbb{E}(C,A)$. Then $x$ is called an {\em inflation} and $y$ is called a {\em deflation}. We say $A\stackrel{x}{\longrightarrow}B\stackrel{y}{\longrightarrow}C\stackrel{\delta}{\dashrightarrow}$ is an $\mathbb{E}$-{\em triangle}. An $\mathbb{E}$-triangle is {\em split} if it realizes $0$.

  (b) For a given $\mathbb{E}$-triangle $A\stackrel{x}{\longrightarrow}B\stackrel{y}{\longrightarrow}C\stackrel{\delta}{\dashrightarrow}$, we denote $A={\rm cocone}(y)$ and $C={\rm cone}(x)$. A subcategory $\mathscr{X}$ of $\mathscr{C}$ is {\em closed under cocones} (resp. {\em cones}) if for any conflation $A\stackrel{x}{\longrightarrow}B\stackrel{y}{\longrightarrow}C$ with $B,C\in\mathscr{X}$ (resp. $A,B\in\mathscr{X}$), we have $A\in\mathscr{X}$ (resp. $C\in\mathscr{X}$).

  (c) A subcategory $\mathscr{X}$ of $\mathscr{C}$ is {\em closed under extensions} if for any conflation $A\stackrel{x}{\longrightarrow}B\stackrel{y}{\longrightarrow}C$ with $A,C\in\mathscr{X}$, we have $B\in\mathscr{X}$.

  (d) An object $P$ in $\mathscr{C}$ is {\em projective} if for any conflation $A\stackrel{x}{\longrightarrow}B\stackrel{y}{\longrightarrow}C$, $\mathscr{C}(P,y)$ is surjective. We denote the subcategory of projective objects by $\mathcal{P}(\mathscr{C})$. Dually, we can define {\em injective} objects and the subcategory of injective objects denoted by $\mathcal{I}(\mathscr{C})$. We say that $\mathscr{C}$ {\em has enough projectives} if for any $A\in\mathscr{C}$, there is a deflation $P\rightarrow A$ for some $P\in\mathcal{P}(\mathscr{C})$. Dually, we define that $\mathscr{C}$ {\em has enough injectives}.

  (e) Let $\mathscr{T}$ be a subcategory of $\mathscr{C}$. Then we define ${\rm Defl}(\mathscr{T})$ and ${\rm Infl}(\mathscr{T})$ as $${\rm Defl}(\mathscr{T})=\{F\in\mathscr{C}\;|\;\exists\;{\rm a}\;{\rm deflation}\; T\stackrel{g}{\longrightarrow}F\;{\rm for}\;{\rm some}\;T\in\mathscr{T}\},$$ $${\rm Infl}(\mathscr{T})=\{S\in\mathscr{C}\;|\;\exists\;{\rm an}\;{\rm inflation}\; S\stackrel{f}{\longrightarrow}T\;{\rm for}\;{\rm some}\;T\in\mathscr{T}\}.$$
\end{remark}

\begin{definition}
  A subcategory $\mathscr{X}\in\mathscr{C}$ is called {\em resolving} if it satisfies the following conditons:

  (1) $\mathcal{P}(\mathscr{C})\subseteq \mathscr{X}$,

  (2) $\mathscr{X}$ is closed under direct summands, cocones and extensions.
\end{definition}

Dually, we have the definition of {\em coresolving~subcategory}.

\subsection{Exact sequences in extriangulated categories}
Throughout the article, we assume $\mathscr{C}$ is an extriangulated category, which has enough projectives and injectives and satisfies the following conditions.

\begin{condition}\label{2.6}
  (WIC). \cite[Condition 5.8]{8}
  (1) For any pair of morphisms $f:X\rightarrow Y$ and $g:Y\rightarrow Z$ in $\mathscr{C}$, if $gf$ is an inflation, then so is $f$.

  (2) For any pair of morphisms $f:X\rightarrow Y$ and $g:Y\rightarrow Z$ in $\mathscr{C}$, if $gf$ is a deflation, then so is $g$.
\end{condition}

\begin{definition}
  \cite[Definition 2.9]{16} A sequence $A\stackrel{x}{\longrightarrow}B\stackrel{y}{\longrightarrow}C$ is said to be {\em right exact $\mathbb{E}$-triangle} if there exists an $\mathbb{E}$-triangle $K\stackrel{h_2}{\longrightarrow}B\stackrel{y}{\longrightarrow}C\stackrel{\delta}{\dashrightarrow}$ and a deflation $h_1:A\rightarrow K$ which is compatible, such that $x=h_2h_1$. Dually, we can define {\em left exact $\mathbb{E}$-triangles}.
\end{definition}

A morphism $f$ in $\mathscr{C}$ is called {\em compatible}, if ``$f$ is both an inflation and a deflation'' implies $f$ is an isomorphism.

\begin{lemma}\label{2.8}
  \cite[Lemma 2.10]{16} Let $\eta:A\stackrel{x}{\longrightarrow}B\stackrel{y}{\longrightarrow}C$ be a right exact $\mathbb{E}$-triangle in $\mathscr{C}$. If x is an inflation, then $\eta$ is a conflation. Dually, we have the similar result on left exact $\mathbb{E}$-triangles.
\end{lemma}

\begin{corollary}
  \cite[Remark 2.11]{16} A sequence $\eta:A\stackrel{x}{\longrightarrow}B\stackrel{y}{\longrightarrow}C$ is both right exact and left exact if and only if $\eta$ is a conflation.
\end{corollary}

\begin{definition}
  \cite[Definition 2.12]{16}Let $(\mathcal{A}, \mathbb{E}_{\mathcal{A}}, \mathfrak{s}_{\mathcal{A}})$ and $(\mathcal{B}, \mathbb{E}_{\mathcal{B}}, \mathfrak{s}_{\mathcal{B}})$ be extriangulated categories. An additive covariant functor $F:\mathcal{A}\longrightarrow\mathcal{B}$ is called a right {\em exact functor} if it satisfies the following conditions:

  (1) If $f$ is a compatible morphism in $\mathcal{A}$, then $Ff$ is compatible in $\mathcal{B}$.

  (2) If $A\stackrel{x}{\longrightarrow}B\stackrel{y}{\longrightarrow}C$ is right exact in $\mathcal{A}$, then $FA\stackrel{Fx}{\longrightarrow}FB\stackrel{Fy}{\longrightarrow}FC$ is right exact in

  \quad\; $\mathcal{B}$. (Then for any $\mathbb{E}_{\mathcal{A}}$-triangle $A\stackrel{x}{\longrightarrow}B\stackrel{y}{\longrightarrow}C\stackrel{\delta}{\dashrightarrow}$, , there exists an $\mathbb{E}_{\mathcal{A}}$-triangle

  \quad\;\,$A'\stackrel{f}{\longrightarrow}FB\stackrel{Fy}{\longrightarrow}FC\dashrightarrow$ such that $Fx=fg$ and $g:FA\longrightarrow A'$ is a deflation and is

  \quad\;\,compatible. Moreover, $A'$ is uniquely determined up to isomorphism.)

  (3) There exists a natural transformation
   $$\eta=\{\eta_{(C,A)}:\mathbb{E}_{\mathcal{A}}(C,A)\longrightarrow\mathbb{E}_{\mathcal{B}}(F^{\rm op}C,A')\}_{(C,A)\in\mathcal{A}^{\rm op}\times\mathcal{A}}$$
  \quad\quad\;\;\,such that $\mathfrak{s}_{\mathcal{B}}(\eta_{(C,A)}(\delta))=[A'\stackrel{f}{\longrightarrow}FB\stackrel{Fy}{\longrightarrow}FC]$.

  Dually, we define the {\em left exact} functor between two extriangulated categories.

  A functor is called {\em exact} is it is both right exact and left exact.
\end{definition}

For any object $C\in\mathscr{C}$ , there exist $\mathbb{E}$-triangles $$A\stackrel{x}{\longrightarrow}P\stackrel{y}{\longrightarrow}C\stackrel{\delta}{\dashrightarrow} \;{\rm and}\; C\stackrel{x}{\longrightarrow}I\stackrel{y}{\longrightarrow}A'\stackrel{\delta}{\dashrightarrow}$$
with $P\in\mathcal{P}(\mathscr{C})$ and $I\in\mathcal{I}(\mathscr{C})$. In this case, $A$ is called the {\em syzygy} and $A'$ is called the {\em cosyszygy} of $C$, which are denoted by $\Omega(C)$ and $\Sigma(C)$, respectively.

For any subcategory $\mathscr{T}$ of $\mathscr{C}$, put $\Omega^0\mathscr{T}=\mathscr{T}$, and for $i>0$ we define $\Omega^i\mathscr{T}$ inductively by $\Omega^i(\mathscr{T})=\Omega(\Omega^{i-1}\mathscr{T})$, i.e. the subcategory consisting of syzygies of objects in $\Omega^{i-1}\mathscr{T}$. We call $\Omega^i\mathscr{T}$ the {\em i}-th syzygy of $\mathscr{T}$. Dually, we can define the {\em i}-th cosyzygy of $\mathscr{T}$ denoted by $\Sigma^i\mathscr{T}$ for $i\geqslant0$.

By \cite[Lemma 5.1]{17}, the higher extension group is defined as
$$\mathbb{E}^{i+1}(X,Y)\cong\mathbb{E}(X,\Sigma^iY)\cong\mathbb{E}(\Omega^iX,Y)$$
for any $X,Y\in\mathscr{C}$ and $i\geqslant0$.

\begin{remark}
  For any subcategory $\mathscr{T}$ of $\mathscr{C}$, we define $$\mathscr{T}^\bot=\{X\in\mathscr{C}~|~\mathbb{E}^i(T,X)=0,~\forall~i\geqslant1,~ T\in\mathscr{T}\}$$
  $$^{\bot}\mathscr{T}=\{X\in\mathscr{C}~|~\mathbb{E}^i(X,T)=0,~\forall~i\geqslant1,~ T\in\mathscr{T}\}$$.
\end{remark}

The following result is well-known in \cite{17} and we will use it frequently in the following.

\begin{proposition}\label{2.13}
   \cite[Proposition 5.2]{17} Let $A\stackrel{x}{\longrightarrow}B\stackrel{y}{\longrightarrow}C\stackrel{\delta}{\dashrightarrow}$ be an $\mathbb{E}$-triangle. There exist long exact sequences $$\ldots\longrightarrow \mathbb{E}^i(X,A)\longrightarrow\mathbb{E}^i(X,B)\longrightarrow\mathbb{E}^i(X,C)$$
  $$\,\,\,\longrightarrow\mathbb{E}^{i+1}(X,A)\longrightarrow\mathbb{E}^{i+1}(X,B)\longrightarrow\ldots $$ and $$\ldots\longrightarrow \mathbb{E}^i(C,X)\longrightarrow\mathbb{E}^i(B,X)\longrightarrow\mathbb{E}^i(A,X)$$
  $$\,\,\,\longrightarrow\mathbb{E}^{i+1}(C,X)\longrightarrow\mathbb{E}^{i+1}(B,X)\longrightarrow\ldots $$ for any objects $X\in\mathscr{C}$ and $i\geqslant0$.
\end{proposition}

Finally, let's recall a definition in \cite{7}. An {\em $\mathbb{E}$-triangle~sequence} in $\mathscr{C}$ is defined as a sequence $$\ldots\longrightarrow X_{n-1}\stackrel{d_{n-1}}{\longrightarrow}X_{n}\stackrel{d_n}{\longrightarrow}X_{n+1}\longrightarrow\ldots$$ such that for any $n$, there are $\mathbb{E}$-triangles $$K_{n}\stackrel{g_n}{\longrightarrow}X_n\stackrel{f_n}{\longrightarrow}K_{n+1}\stackrel{\delta_n}{\dashrightarrow}$$ in $\mathscr{C}$ and the differential $d_n=g_{n+1}f_n$.

A 4-term $\mathbb{E}$-triangle sequence $A\stackrel{x}{\longrightarrow}B\stackrel{y}{\longrightarrow}C\stackrel{z}{\longrightarrow}D$ is called {\em right~exact} (resp. {\em left~exact}) if there are $\mathbb{E}$-triangles $A\stackrel{x}{\longrightarrow}B\stackrel{y_1}{\longrightarrow}K\dashrightarrow$ and $K\stackrel{y_2}{\longrightarrow}C\stackrel{z}{\longrightarrow}D\dashrightarrow$ such that $y_2y_1=y$ and $y_1$ (resp. $y_2$) is compatible.

The following, which is similar to the famous Horseshoe Lemma, plays a very important role in the proof of Section 4.

\begin{lemma}\label{2.12}
  Let $A\stackrel{x}{\longrightarrow}B\stackrel{y}{\longrightarrow}C$ be a conflation in $\mathscr{C}$.

  {\rm (1)} Assume that $$\xymatrix@R=0.5ex{A\ar[r]^{d^0_A}&A^0\ar[rr]^{d^1_A}\ar[rd]&&A^1\ar[rr]^{d^2_A}\ar[rd]&&A^2\ar[r]&\ldots\;;\\
  &&K^1_A\ar[ru]&&K^2_A\ar[ru]}
  $$$$\xymatrix@R=0.5ex{C\ar[r]^{d^0_C}&C^0\ar[rr]^{d^1_C}\ar[rd]&&C^1\ar[rr]^{d^2_C}\ar[rd]&&C^2\ar[r]&\ldots\\
  &&K^1_C\ar[ru]&&K^2_C\ar[ru]}$$ are two $\mathbb{E}$-triangle sequences in $\mathscr{C}$. If $\mathbb{E}(C,A^0)=0$ and $\mathbb{E}(K^i_C,A^i)=0$ for any $i\geqslant1$, there is a commutative diagram
  $$\xymatrix{A\ar[r]^{x}\ar[d]_{d^0_A}&B\ar[r]^{y}\ar[d]&C\ar[d]^{d^0_C}\\
  A^0\ar[r]^{\tiny{\begin{pmatrix}1\\0\end{pmatrix}}\quad\quad}\ar[d]_{d^1_A}&A^0\oplus C^0\ar[r]^{\quad\tiny{\setlength{\arraycolsep}{1.2pt}\begin{pmatrix}0 & 1\end{pmatrix}}}\ar[d]&C^0\ar[d]^{d^1_C}\\
  A^1\ar[r]^{\tiny{\begin{pmatrix}1\\0\end{pmatrix}}\quad\quad}\ar[d]_{d^2_A}&A^1\oplus C^1\ar[r]^{\quad\tiny{\setlength{\arraycolsep}{1.2pt}\begin{pmatrix}0 & 1\end{pmatrix}}}\ar[d]&C^1\ar[d]^{d^2_C}\\
  \vdots&\vdots&\vdots}
  $$ where rows are conflations and columns are $\mathbb{E}$-triangle sequences.

  (2) Assume $$\xymatrix@R=0.5ex{\ldots\ar[r]&A_2 \ar[rr]^{d_2^A}\ar[rd]&&A_1\ar[rr]^{d_1^A}\ar[rd]&&A_0\ar[r]^{d_0^A}& A\;;\\
  &&K_2^A\ar[ru]&&K_1^A\ar[ru]}
  $$$$\xymatrix@R=0.5ex{\ldots\ar[r]&C_2 \ar[rr]^{d_2^C}\ar[rd]&&C_1\ar[rr]^{d_1^C}\ar[rd]&&C_0\ar[r]^{d_0^C}& C\\
  &&K_2^C\ar[ru]&&K_1^C\ar[ru]}$$ are two $\mathbb{E}$-triangle sequences in $\mathscr{C}$. If $\mathbb{E}(C_0,A)=0$ and $\mathbb{E}(C_i,K^A_i)=0$ for any $i\geqslant1$, there is a commutative diagram

  $$\xymatrix{\vdots\ar[d]_{d^0_A}&\vdots\ar[d]&\vdots\ar[d]^{d^0_C}\\
  A_1\ar[r]^{\tiny{\begin{pmatrix}1\\0\end{pmatrix}}\quad\quad}\ar[d]_{d_2^A}&A_1\oplus C_1\ar[r]^{\quad\tiny{\setlength{\arraycolsep}{1.2pt}\begin{pmatrix}0 & 1\end{pmatrix}}}\ar[d]&C_1\ar[d]^{d_2^C}\\
  A_0\ar[r]^{\tiny{\begin{pmatrix}1\\0\end{pmatrix}}\quad\quad}\ar[d]_{d_1^A}&A_0\oplus C_0\ar[r]^{\quad\tiny{\setlength{\arraycolsep}{1.2pt}\begin{pmatrix}0 & 1\end{pmatrix}}}\ar[d]&C_0\ar[d]^{d_1^C}\\
  A\ar[r]^{x}&B\ar[r]^{y}&C}
  $$ where rows are conflations and columns are $\mathbb{E}$-triangle sequences.
\end{lemma}

\begin{proof}
  We only prove (1), (2) can be proved similarly.

  By \cite[Lemma 4.15]{18}, there is an inflation $d^0_B:B\longrightarrow A_0\oplus C_0$ which makes the following diagram commutative since $\mathbb{E}(C,A^0)=0$.
  $$\xymatrix{A\ar[r]^{x}\ar[d]_{d^0_A}&B\ar[r]^{y}\ar[d]^{d^0_B}\ar[d]&C\ar[d]^{d^0_C}\\
  A^0\ar[r]^{\tiny{\begin{pmatrix}1\\0\end{pmatrix}}\quad\quad}&A^0\oplus C^0\ar[r]^{\quad\tiny{\setlength{\arraycolsep}{1.2pt}\begin{pmatrix}0 & 1\end{pmatrix}}}&C^0 }$$
  We complete $d^0_B$ by the conflation $$B\stackrel{d^0_B}{\longrightarrow}A_0\oplus C_0 \longrightarrow K^B_1$$ for some $K^B_1\in\mathscr{C}$. By the duality of \cite[Lemma 4.14]{18}, there is a conflation $K^A_1\stackrel{x_1}{\longrightarrow}K^B_1\stackrel{y_1}{\longrightarrow}K^C_1$ which makes the following diagram commutative
  $$\xymatrix{A\ar[r]^{x}\ar[d]_{d^0_A}&B\ar[r]^{y}\ar[d]^{d^0_B}&C\ar[d]^{d^0_C}\\
  A^0\ar[r]^{\tiny{\begin{pmatrix}1\\0\end{pmatrix}}\quad\quad}\ar[d]&A^0\oplus C^0\ar[r]^{\quad\tiny{\setlength{\arraycolsep}{1.2pt}\begin{pmatrix}0 & 1\end{pmatrix}}}\ar[d]&C^0\ar[d]\\
  K^A_1\ar[r]^{x_1}&K^B_1\ar[r]^{y_1}&K^A_1}$$ where each row and column is a conflation. Since $\mathbb{E}(K^i_C,A^{i+1})=0$ for any $i\geqslant1$, we can repeat the construction as above. Therefore, we complete the proof.
\end{proof}

\section{Wakamatsu-tilting subcategories}

In this section, we will begin with the definitions of Wakamatsu-tilting subcategories and Wakamatsu-cotilting subcategories in a extriangulated category. Then we will show that these two definitions and the definitions of $\infty$-tilting subcategories and $\infty$-cotilting subcategories in \cite{15} are equivalent.

For a subcategory $\mathscr{T}\subseteq\mathscr{C}$, denote by ${_{\mathscr{T}}\mathcal{X}}$ the subcategory of all objects $M\in\mathscr{C}$ such that there is an infinite $\mathbb{E}$-triangle sequence
$$\xymatrix@R=0.5ex{\ldots\ar[r]&T_2^M\ar[rr]^{d_2^M}\ar[rd]&&T_1^M\ar[rr]^{d_1^M}\ar[rd]&&T_0^M\ar[r]^{d_0^M}&M\\
&&K_1^M\ar[ru]&&K_0^M\ar[ru]}$$
with $M,K_i^M\in\mathscr{T}^{\bot},T_i^M\in\mathscr{T}$ for $i\geqslant0$. Dually, we denote by $\mathcal{X}_{\mathscr{T}}$ the subcategory of all objects $N\in\mathscr{C}$ such that there is an infinite $\mathbb{E}$-triangle sequence
$$\xymatrix@R=0.5ex{N\ar[r]^{d^0_N}&T^0_N\ar[rr]^{d^1_N}\ar[rd]&&T^1_N\ar[rr]^{d^2_N}\ar[rd]&&T^2_N\ar[r]&\ldots\\
&&C^0_N\ar[ru]&&C^1_N\ar[ru]}$$
with $N,C^i_N\in{^{\bot}\mathscr{T}},T^i_N\in\mathscr{T}$ for $i\geqslant0$.

Let $\mathscr{X}\subseteq\mathscr{Y}$ be two subcategories of $\mathscr{C}$. The subcategory $\mathscr{X}$ is called an {\em $\mathbb{E}$-projective~generator} of $\mathscr{Y}$ if $\mathbb{E}(X,Y)=0$ for any $X\in\mathscr{X},Y\in\mathscr{Y}$ and for any $Y\in\mathscr{Y}$, there is a conflation $Y_1\longrightarrow X\longrightarrow Y$ with $X\in\mathscr{X},Y_1\in\mathscr{Y}$. Dually, we can define the {\em $\mathbb{E}$-injective~cogenerator}.

From now on, each subcategory $\mathscr{T}$ is closed under direct summands.

\begin{definition}\label{3.1}
  \cite[Definition 3.1]{12} Let $\mathscr{T}$ be a subcategory of $\mathscr{C}$. $\mathscr{T}$ is called a {\em Wakamatsu-tilting~subcategory} if the following conditions are satisfied:

  (WT1) $\mathbb{E}^n(\mathscr{T},\mathscr{T})=0$ for any $n\geqslant1$,

  (WT2) $\mathcal{P}(\mathscr{C})\subseteq\mathcal{X}_{\mathscr{T}}$.
\end{definition}

Dually, we have the definition of Wakamatsu-cotilting subcategories.

\begin{definition}\label{3.2}
  Let $\mathscr{T}$ be a subcategory of $\mathscr{C}$. $\mathscr{T}$ is called a {\em Wakamatsu-cotilting subcategory} if the following conditions are satisfied:

  (WC1) $\mathbb{E}^n(\mathscr{T},\mathscr{T})=0$ for any $n\geqslant1$,

  (WC2) $\mathcal{I}(\mathscr{C})\subseteq {_{\mathscr{T}}\mathcal{X}}$.
\end{definition}

Zhang, Wei and Wang defined $\infty$-tilting subcategories and $\infty$-cotilting subcategories in an extriangulated category. In fact, these four subcategories are equivalent.

\begin{definition}\label{3.3}
  \cite[Definition 3.1]{15} Let $\mathscr{T}$ be a subcategory of $\mathscr{C}$. $\mathscr{T}$ is called an $\infty$-{\em tilting~subcategory} if the following conditions are satisfied:

  ($\infty$T1) $\mathscr{T}$ is an $\mathbb{E}$-projective generator of ${_{\mathscr{T}}\mathcal{X}}$,

  ($\infty$T2) $\mathcal{I}(\mathscr{C})\subseteq{_{\mathscr{T}}\mathcal{X}}$.
\end{definition}

\begin{definition}\label{3.4}
  \cite[Definition 3.2]{15} Let $\mathscr{T}$ be a subcategory of $\mathscr{C}$. $\mathscr{T}$ is called an $\infty$-{\em cotilting~subcategory} if the following conditions are satisfied:

  ($\infty$C1) $\mathscr{T}$ is an $\mathbb{E}$-injective cogenerator of $\mathcal{X}_{\mathscr{T}}$,

  ($\infty$C2) $\mathcal{P}(\mathscr{C})\subseteq\mathcal{X}_{\mathscr{T}}$.
\end{definition}

In order to show that these four definitions coincide with each other, we need the following lemma.

\begin{lemma}\label{3.5}
  Let $\mathscr{T}$ be a subcategory of $\mathscr{C}$ with $\mathbb{E}^n(\mathscr{T},\mathscr{T})=0$ for any $n\geqslant1$. Then $\mathcal{P}(\mathscr{C})\subseteq\mathcal{X}_{\mathscr{T}}$ if and only if $\mathcal{I}(\mathscr{C})\subseteq{_{\mathscr{T}}\mathcal{X}}$.
\end{lemma}

\begin{proof}
  We only prove the necessity; the sufficiency can be proved dually.

  We claim that ${_{\mathscr{T}}\mathcal{X}}=\mathscr{S}^{\bot}$ for some subcategory $\mathscr{S}\subseteq\mathscr{C}$. Firstly, we construct the $\mathscr{S}$. For any $P\in\mathcal{P}(\mathscr{C})\subseteq\mathcal{X}_{\mathscr{T}}$, we fix an infinite $\mathbb{E}$-triangle sequence
  $$\xymatrix@R=0.5ex{P\ar[r]^{d^0_P}&T^0_P\ar[rr]^{d^1_P}\ar[rd]&& T^1_P\ar[rr]^{d^2_P}\ar[rd]&&T^2_P\ar[r]&\ldots\\
  &&C^0_P\ar[ru]&&C^1_P\ar[ru]}$$ with $P,C^i_P\in{^{\bot}\mathscr{T}},T^i_P\in\mathscr{T}$ for $i\geqslant0$. We define $\bar{\mathscr{S}}=\{C^i_P~|~P\in\mathcal{P}(\mathscr{C}),~i\geqslant0\}$. Then $\mathscr{S}=\mathscr{T}\cup\bar{\mathscr{S}}$.

  By the definition of ${_{\mathscr{T}}\mathcal{X}}$, it is clear that $\mathscr{T}\subseteq{^{\bot}{_{\mathscr{T}}\mathcal{X}}}$. It suffices to show that $\mathbb{E}^j(C^i_P,M)=0$ for any $M\in{_{\mathscr{T}}\mathcal{X}},P\in\mathcal{P}(\mathscr{C}),j\geqslant1,i\geqslant0$. Consider the infinite $\mathbb{E}$-triangle sequence
  $$\xymatrix@R=0.5ex{\ldots\ar[r]&T_2^M\ar[rr]^{d_2^M}\ar[rd]&&T_1^M\ar[rr]^{d_1^M}\ar[rd]&&T_0^M\ar[r]^{d_0^M}&M.\\
  &&K_1^M\ar[ru]&&K_0^M\ar[ru]}$$
  $M,K_i^M\in\mathscr{T}^{\bot},T_i^M\in\mathscr{T}$ for $i\geqslant0$.
  Clearly, $K_i^M\in{_{\mathscr{T}}\mathcal{X}}$ for any $i\geqslant0$. By applying the functor $\mathscr{C}(-,M)$ to the conflation $P\stackrel{d^0_P}{\longrightarrow}T^0_P\longrightarrow C^0_P$, we have a long exact sequence
  $$\ldots\longrightarrow\mathscr{C}(T^0_P,M)\stackrel{\mathscr{C}(d^0_P,M)}{\longrightarrow}\mathscr{C}(P,M) \longrightarrow \mathbb{E}(C^0_P,M)\longrightarrow\mathbb{E}(T^0_P,M)=0
  $$since $M\in\mathscr{T}^{\bot}$. For any morphism $g:P\longrightarrow M$, $g$ lifts to a morphism $g':P\longrightarrow T_0^M$ and $g=d_0^Mg'$. By applying the functor $\mathscr{C}(-,T_0^M)$ to the conflaton $P\stackrel{d^0_P}{\longrightarrow}T^0_P\longrightarrow C^0_P$, we have a long exact sequence $$\ldots\longrightarrow\mathscr{C}(T^0_P,T_0^M)\stackrel{\mathscr{C}(d^0_P,T_0^M)}{\longrightarrow}\mathscr{C}(P,T_0^M) \longrightarrow \mathbb{E}(C^0_P,T_0^M)=0$$ since $C^0_P\in{^{\bot}\mathscr{T}}$ and $T_0^M\in\mathscr{T}$. Since $\mathscr{C}(d^0_P,T_0^M)$ is an epimorphism, $g'$ extend to a morphism $g'':T^0_P\longrightarrow T_0^M$ such that $g''d^0_P=g'$. Then $d^M_0g''d^0_P=d^M_0g'=g$, which induce that $\mathscr{C}(d^0_P,M)$ is surjective. Hence, $\mathbb{E}(C^0_P,M)=0$ for any $M\in{_{\mathscr{T}}\mathcal{X}}$. By applying the functor $\mathscr{C}(C_P^{t+1},-)$ to the conflation $K_0^M\longrightarrow T^M_0\stackrel{d^M_0}{\longrightarrow}M$, we have a long exact sequence $$0=\mathbb{E}^j(C_P^{t+1},T^M_0)\longrightarrow\mathbb{E}^j(C_P^{t+1},M) \stackrel{\sim}{\longrightarrow} \mathbb{E}^{j+1}(C_P^{t+1},K_0^M)\longrightarrow\mathbb{E}^{j+1}(C_P^{t+1},T^M_0)=0$$ for any $t\geqslant0,j\geqslant1$ since $C_P^{t+1}\in{^{\bot}\mathscr{T}}$ and $T_0^M\in\mathscr{T}$. By applying the functor $\mathscr{C}(-,K_0^M)$ to the conflation $C_P^t\longrightarrow T^{t+1}_P\longrightarrow C_P^{t+1}$, we have a long exact sequence $$0=\mathbb{E}^j(T^{t+1}_P,K_0^M)\longrightarrow\mathbb{E}^j(C_P^t,K_0^M) \stackrel{\sim}{\longrightarrow} \mathbb{E}^{j+1}(C_P^{t+1},K_0^M)\longrightarrow\mathbb{E}^{j+1}(T^{t+1}_P,K_0^M)=0$$ for any $t\geqslant0,j\geqslant1$ since $K_0^M\in{\mathscr{T}^{\bot}}$ and $T^{t+1}_P\in\mathscr{T}$. We obtain that $\mathbb{E}^j(C_P^{t+1},M)\cong\mathbb{E}^{j+1}(C_P^{t+1},K_0^M)\cong\mathbb{E}^j(C_P^t,K_0^M)$. By induction, we have $\mathbb{E}(C^i_P,M)\cong\mathbb{E}(C^0_P,K^M_{i-1})=0$ for any $i\geqslant1$ since $K^M_{i-1}\in{_{\mathscr{T}}\mathcal{X}}$. Then we have $$\mathbb{E}^j(C^i_P,M)=\left\{\begin{array}{cc}
                                        \mathbb{E}^{j-i-1}(P,M)=0, & j\geqslant i+2 \\
                                        \mathbb{E}(C^{i-j+1}_P,M)=0, & j\leqslant i+1.
                                      \end{array}\right.$$
  for any $M\in{_{\mathscr{T}}\mathcal{X}},P\in\mathcal{P}(\mathscr{C}),i\geqslant0,j\geqslant1$. Since $\bar{\mathscr{S}}\cup\mathscr{T}\subseteq{^{\bot}{_{\mathscr{T}}\mathcal{X}}}$, we have ${_{\mathscr{T}}\mathcal{X}}\subseteq{(\bar{\mathscr{S}}\cup\mathscr{T})^{\bot}}= {\mathscr{S}^{\bot}}$.

  Conversely, let $M\in{\mathscr{S}^{\bot}}$. Since $\mathscr{C}$ has enough projectives, there is a conflation $N\longrightarrow Q\longrightarrow M$ for some $Q\in\mathcal{P}(\mathscr{C})$. Consider the conflation $Q\stackrel{d^0_Q}{\longrightarrow}T^0_Q\longrightarrow C^0_Q$, there is a commutative diagram of conflations
  $$\xymatrix{N \ar[r] \ar@{=}[d]& Q\ar[r]\ar[d]&M\ar[d]\\
  N\ar[r]&T^0_Q\ar[r]\ar[d]&E\ar[d]\\
  &C^0_Q\ar@{=}[r]&C^0_P.}$$
  Thus $E\in{\rm Defl}(\mathscr{T})$. Since $M\in{\mathscr{S}^{\bot}}$, the third column is a split conflation and $M\in{\rm Defl}(\mathscr{T})$. Then there is a deflation $g:T_1\longrightarrow M$ for some $T_1\in \mathscr{T}$. Consider the set of morphisms $\Lambda=\{f:T^f\longrightarrow M~|~T^f\in \mathscr{T} \}$. Take $T'=\bigoplus\limits_{f\in\Lambda}T^f\in\mathscr{T}$ and we have a morphism $G:T'\longrightarrow D$ naturally. Now since $g\in \Lambda$ and it clearly factors through the $G$, $G:T' \longrightarrow M$ is a deflation by Condition WIC. Therefore, we get a conflation $$M_1\longrightarrow T'\stackrel{G}{\longrightarrow} M$$ with $G$ a right $\mathscr{T}$-approximation of $M$. Since $M\in{\mathscr{S}^{\bot}}\subseteq{\mathscr{T}^{\bot}}$, also $M_1\in{\mathscr{T}^{\bot}}$ by Proposition \ref{2.13}. By applying the functor $\mathscr{C}(C_P^{t+1},-)$ to the conflation $M_1\longrightarrow T'\longrightarrow M$, we have a long exact sequence $$0=\mathbb{E}^j(C_P^{t+1},T')\longrightarrow\mathbb{E}^j(C_P^{t+1},M) \stackrel{\sim}{\longrightarrow} \mathbb{E}^{j+1}(C_P^{t+1},M_1)\longrightarrow\mathbb{E}^{j+1}(C_P^{t+1},T')=0$$ for any $P\in\mathcal{P}(\mathscr{C}), t\geqslant0,j\geqslant1$ since $C_P^{t+1}\in{^{\bot}\mathscr{T}}$ and $T'\in\mathscr{T}$. By applying the functor $\mathscr{C}(-,M_1)$ to the conflation $C_P^t\longrightarrow T^{t+1}_P\longrightarrow C_P^{t+1}$, we have a long exact sequence $$0=\mathbb{E}^j(T^{t+1}_P,M_1)\longrightarrow\mathbb{E}^j(C_P^t,M_1) \stackrel{\sim}{\longrightarrow} \mathbb{E}^{j+1}(C_P^{t+1},M_1)\longrightarrow\mathbb{E}^{j+1}(T^{t+1}_P,M_1)=0$$ for any $t\geqslant0,j\geqslant1$ since $M_1\in{\mathscr{T}^{\bot}}$ and $T^{t+1}_P\in\mathscr{T}$. We obtain that $\mathbb{E}^j(C_P^t,M_1)\cong\mathbb{E}^{j+1}(C_P^{t+1},M_1)\cong\mathbb{E}^j(C_P^{t+1},M)=0$ since $M\in{\mathscr{S}^{\bot}}$. Then $M_1\in{\mathscr{S}^{\bot}}$. Repeating the same argument for $M_1$ and so on, we obtain that $M\in{_{\mathscr{T}}\mathcal{X}}$. Hence, we complete the proof of the claim.

  Clearly, $\mathcal{I}(\mathscr{C})\subseteq{\mathscr{S}^{\bot}}={_{\mathscr{T}}\mathcal{X}}$ and the proof is finished.
\end{proof}

\begin{theorem}\label{3.6}
  The definitions from Definition \ref{3.1} to Definition \ref{3.4} are equivalent. In other words, we obtain the following diagram
  $$
 \xymatrix@R=3.5em{{\rm Wakamatsu}-{\rm tilting~subcategory}\ar@{<=>}[r]\ar@{<=>}[d]& \infty{\rm -tilting~subcategory}\ar@{<=>}[d]
 \\{\rm Wakamatsu-cotilting~subcategory}\ar@{<=>}[r]&\infty{\rm -cotilting~subcategory}.}
$$
\end{theorem}

\begin{proof}
  By the definitions of ${_{\mathscr{T}}\mathcal{X}}$ and $\mathcal{X}_{\mathscr{T}}$, $\mathbb{E}^n(\mathscr{T},\mathscr{T})=0$ for any $n\geqslant1$ $\Longleftrightarrow\mathscr{T}$ is an $\mathbb{E}$-projective generator of ${_{\mathscr{T}}\mathcal{X}}\Longleftrightarrow\mathscr{T}$ is an $\mathbb{E}$-injective cogenerator of $\mathcal{X}_{\mathscr{T}}$. By Lemma \ref{3.5}, the proof is finished.
\end{proof}

In \cite{15}, Zhang shows that there is an one-to-one correspondence between $\infty$-tilting (resp. $\infty$-cotilting) subcategories and coresolving (resp. resolving) subcategories with an $\mathbb{E}$-projective generator (resp. $\mathbb{E}$-injective cogenerator), which satisfy some conditions. Now, we will obtain the following corollary.

\begin{corollary}
  {\rm(1)} There is an inverse bijection between classes of Wakamatsu-tilting subcategories $\mathscr{T}$ and coresolving subcategories $\mathscr{X}$ with an $\mathbb{E}$-projective generator, maximal among those with the same $\mathbb{E}$-projective generator, and the assignments are $\phi:\mathscr{T}\mapsto{_{\mathscr{T}}\mathcal{X}}$ and $\Phi:\mathscr{X}\mapsto{^{\bot}\mathscr{X}}\cap\mathscr{X}$.

  {\rm(2)} There is an inverse bijection between classes of Wakamatsu-tilting subcategories $\mathscr{T}$ and resolving subcategories $\mathscr{X}$ with an $\mathbb{E}$-injective cogenerator, maximal among those with the same $\mathbb{E}$-injective cogenerator, and the assignments are $\phi:\mathscr{T}\mapsto{\mathcal{X}_{\mathscr{T}}}$ and $\Phi:\mathscr{X}\mapsto{\mathscr{X}^{\bot}}\cap\mathscr{X}$.
\end{corollary}

\begin{proof}
  The proof is clear by Theorem 3.3 in \cite{15} and Lemma \ref{3.6}.
\end{proof}
\section{Recollement}

In \cite{16}, Wang, Wei and Zhang defined the recollement $(\mathcal{A},\mathcal{B},\mathcal{C})$ of extriangulated categories. In this section, we will glue together Wakamatsu-tilting subcategories in $\mathcal{A}$ and $\mathcal{C}$, to obtain a Wakamatsu-tilting subcategory of $\mathcal{B}$ in a special recollement. Conversely, we will give a class of specail Wakamatsu-tilting subcategories, which can induce the Wakamatsu-tilting subcategories in $\mathcal{A}$ and $\mathcal{C}$ by the involved functors in the gluing. Throughout the section, we assume $\mathcal{A},\mathcal{B}$ and $\mathcal{C}$ are extriangulated categories, which have enough projectives and injectives and satisfy (WIC) conditions. We begin the section with the definition of recollement of extriangulated categories.

\begin{definition}
  \cite[Definition 3.1]{16} A {\em recollement} of $\mathcal{B}$ relative to $\mathcal{A}$ and $\mathcal{C}$, denoted by $(\mathcal{A},\mathcal{B},\mathcal{C})$, is a diagram
  $$\xymatrix{\mathcal{A}\ar[rr]|{i_{\ast}}&&
  \mathcal{B}\ar@/_1pc/[ll]|{i^{\ast}}\ar@/^1pc/[ll]|{i^!}\ar[rr]|{j^{\ast}}&&
  \mathcal{C}\ar@/_1pc/[ll]|{j_!}\ar@/^1pc/[ll]|{j_{\ast}}
  }$$
  given by two exact functors $i_{\ast},j^{\ast}$,  two right exact functors $i^{\ast},j_!$ and two left exact functors $i^!,j_{\ast}$, which satisfies the following conditions:

  (R1) $(i^{\ast},i_{\ast},i^!)$ and $(j_!,j^{\ast},j_{\ast})$ are adjoint triples,

  (R2) ${\rm Im}i_{\ast}={\rm Ker}j^{\ast}$,

  (R3) $i_{\ast},j_!$ and $j_{\ast}$ are fully faithful,

  (R4) For each $X\in\mathcal{B}$,  there exists a left exact $\mathbb{E}_{\mathcal{B}}$-triangle sequence
  $$i_{\ast}i^!X\stackrel{\theta_X}{\longrightarrow}X\stackrel{\vartheta_X}
  {\longrightarrow}j_{\ast}j^{\ast}X\longrightarrow i_{\ast}A$$
  \quad\quad\,\,\,\;\; with $A\in\mathscr{A}$, where $\theta_X$ and $\vartheta_X$ are given by the adjunction morphisms.

  (R5) For each $X\in\mathcal{B}$,  there exists a right exact $\mathbb{E}_{\mathcal{B}}$-triangle sequence
  $$i_{\ast}A'\longrightarrow j_!j^{\ast}X\stackrel{\upsilon_X}{\longrightarrow}
  X\stackrel{\nu_X}{\longrightarrow}i_{\ast}i^{\ast}X$$
  \quad\quad\,\,\,\;\; with $A'\in\mathscr{A}$, where $\upsilon_X$ and $\nu_X$ are given by the adjunction morphisms.
\end{definition}

Moreover, a recollement $(\mathcal{A},\mathcal{B},\mathcal{C})$ is called $\theta$-$\nu$-{\em compatible}, if the $\theta_X$ in (R4) and the $\nu_X$ in (R5)  are compatible. We will glue the Wakamatsu-tilting subcategories in such special recollements.

We collect some properties of recollements which are very useful later on.

\begin{lemma}\label{4.2}
  \cite[Lemma 3.3]{16} Let $(\mathcal{A},\mathcal{B},\mathcal{C})$ be a recollement of extriangulated categories.

  {\rm(1)} All the natural transformations
  $$i^{\ast}i_{\ast}\Rightarrow {\rm Id}_{\mathcal{A}},\quad\;
  {\rm Id}_{\mathcal{A}}\Rightarrow i^!i_{\ast},\quad\;
  {\rm Id}_{\mathcal{C}}\Rightarrow j^{\ast}j_!,\quad\;
  j^{\ast}j_{\ast}\Rightarrow {\rm Id}_{\mathcal{C}}$$
  \quad\quad\; are natural isomorphisms.

  {\rm(2)} $i^{\ast}j_!=0$ and $i^!j_{\ast}=0$.

  {\rm(3)} $i^{\ast}$ preserves projective objects and $i^!$ preserves injective objects.

  {\rm(3$'$)} $j_!$ preserves projective objects and $j_{\ast}$ preserves injective objects.

  {\rm(4)} If $i^!$ (resp. $j_{\ast}$) is exact, then $i_{\ast}$ (resp. $j^{\ast}$) preserves projective objects.

  {\rm(4$'$)} If $i^{\ast}$ (resp. $j_!$) is exact, then $i_{\ast}$ (resp. $j^{\ast}$) preserves injective objects.

  {\rm(5)} If $i^{\ast}$ is exact, then $j_!$ is exact.

  {\rm(5$'$)} If $i^!$ is exact, then $j_{\ast}$ is exact.
\end{lemma}

\begin{proposition}\label{4.3}
  \cite[Proposition 3.4]{16} Let $(\mathcal{A},\mathcal{B},\mathcal{C})$ be a recollement of extriangulated categories.

  {\rm(1)} If $i^!$ is exact, then for each $X\in\mathcal{B}$,  there exists an $\mathbb{E}_{\mathcal{B}}$-triangle
  $$i_{\ast}i^!X\stackrel{\theta_X}{\longrightarrow}X\stackrel{\vartheta_X}
  {\longrightarrow}j_{\ast}j^{\ast}X\dashrightarrow$$
  \quad\quad\; where $\theta_X$ and $\vartheta_X$ are given by the adjunction
  morphisms.

  {\rm(2)} If $i^{\ast}$ is exact, then for each $X\in\mathcal{B}$,  there exists an $\mathbb{E}_{\mathcal{B}}$-triangle
  $$j_!j^{\ast}X\stackrel{\upsilon_X}{\longrightarrow}
  X\stackrel{\nu_X}{\longrightarrow}i_{\ast}i^{\ast}X\dashrightarrow$$
  \quad\quad\;where $\upsilon_X$ and $\nu_X$ are given by the adjunction morphisms.
\end{proposition}

\begin{lemma}
  Let $(\mathcal{A},\mathcal{B},\mathcal{C})$ be a recollement of extriangulated categories and $n$ be any positive integer.

  {\rm(1)} $i^{\ast}$ is exact, then $\mathbb{E}^n_{\mathcal{A}}(i^{\ast}X,Y)\cong\mathbb{E}^n_{\mathcal{B}}(X,i_{\ast}Y)$ for any $X\in\mathcal{B}$ and $Y\in\mathcal{A}$.

  {\rm(2)} $i^!$ is exact, then $\mathbb{E}^n_{\mathcal{B}}(i_{\ast}X,Y)\cong\mathbb{E}^n_{\mathcal{B}}(X,i^!Y)$ for any $X\in\mathcal{A}$ and $Y\in\mathcal{B}$.

  {\rm(3)} $j_!$ is exact, then $\mathbb{E}^n_{\mathcal{B}}(j_!X,Y)\cong\mathbb{E}^n_{\mathcal{C}}(X,j^{\ast}Y)$ for any $X\in\mathcal{C}$ and $Y\in\mathcal{B}$.

  {\rm(4)} $j_{\ast}$ is exact, then $\mathbb{E}^n_{\mathcal{C}}(j^{\ast}X,Y)\cong\mathbb{E}^n_{\mathcal{B}}(X,j_{\ast}Y)$ for any $X\in\mathcal{B}$ and $Y\in\mathcal{C}$.
\end{lemma}

\begin{proof}
  We only prove (1), others can be proved similarly.

  For any $X\in\mathcal{B}$, there exists a conflation
  \begin{equation}\label{conf1}
    X'\longrightarrow P\longrightarrow X
  \end{equation} for some $P\in\mathcal{P}(\mathcal{B})$ and $X'\in\mathcal{B}$ since $\mathcal{B}$ has enough projectives. By applying $i^{\ast}$ to (\ref{conf1}), we obtain a conflation in $\mathcal{A}$
  \begin{equation}\label{conf2}
    i^{\ast}X'\longrightarrow i^{\ast}P\longrightarrow i^{\ast}X
  \end{equation}
  Since $i^{\ast}$ preserves projective objects by Lemma \ref{4.2}, we have $\mathbb{E}^n_{\mathcal{B}}(i^{\ast}P,Y)\cong \mathbb{E}^n_{\mathcal{A}}(P,i_{\ast}Y)=0$ for any $Y\in\mathcal{A}$ and $n\geqslant1$. By applying ${\mathcal{B}}(-,i_{\ast}Y)$ to (\ref{conf1}) and ${\mathcal{A}}(-,Y)$ to (\ref{conf2}), we obtain two commutative diagrams of exact sequences
  $$\xymatrix{\mathcal{B}(P,i_{\ast}Y)\ar[d]^{\sim}\ar[r]
  &\mathcal{B}(X',i_{\ast}Y)\ar[d]^{\sim}\ar[r]&\mathbb{E}_{\mathcal{B}}(X,i_{\ast}Y)\ar[d]\ar[r]&0\\
  \mathcal{A}(i^{\ast}P,Y)\ar[r]&\mathcal{A}(i^{\ast}X',Y)\ar[r]&\mathbb{E}_{\mathcal{A}}(i^{\ast}X,Y)\ar[r]&0\;;
  }$$
  $$\xymatrix{0\ar[r]&\mathbb{E}^i_{\mathcal{B}}(X',i_{\ast}Y)\ar[d]^{\sim}\ar[r]^{\sim} &\mathbb{E}^{i+1}_{\mathcal{B}}(X,i_{\ast}Y)\ar[d]\ar[r]&0\\
  0\ar[r]&\mathbb{E}^i_{\mathcal{A}}(i^{\ast}X',Y)\ar[r]^{\sim}&\mathbb{E}^{i+1}_{\mathcal{A}}(i^{\ast}X,Y)\ar[r]&0\;;
  }$$
  for any $i\geqslant1$. By five lemma, we obtain that $\mathbb{E}_{\mathcal{A}}(i^{\ast}X,Y)\cong\mathbb{E}_{\mathcal{B}}(X,i_{\ast}Y)$ for any $X\in\mathcal{B}$ and $Y\in\mathcal{A}$. By induction, we have $\mathbb{E}^n_{\mathcal{A}}(i^{\ast}X,Y)\cong\mathbb{E}^n_{\mathcal{B}}(X,i_{\ast}Y)$ and the proof is finished.
\end{proof}

The following Lemma is essential in the proof of the final result.

\begin{lemma}\label{4.5}
  Let $(\mathcal{A},\mathcal{B},\mathcal{C})$ be a $\theta$-$\nu$-compatible recollement of extriangulated categories.

  {\rm(1)} If $i^{\ast}$ is exact, then $i^!j_!=0$.

  {\rm(2)} If $i^!$ is exact, then $i^{\ast}j_{\ast}=0$.
\end{lemma}

\begin{proof}
  We only prove (1), (2) can be proved similarly.

  For any $Y\in\mathcal{C}$, $j_!Y$ is in $\mathcal{B}$. Then there exists a left exact $\mathbb{E}_{\mathcal{B}}$-triangle sequence
  $$\xymatrix@R=0.5ex{i_{\ast}i^!j_!Y\ar[r]^{\;\;\;\theta_{j_!Y}}&j_!Y\ar[rr]^{\vartheta_{j_!Y}\;\;}\ar[rd]|{h_1}&&j_{\ast}j^{\ast}j_!Y\ar[r]&i_{\ast}A.\\
  &&M\ar[ru]|{h_2}}$$
  with $A\in\mathcal{A}$, where $\theta_{j_!Y}$ and $\vartheta_{j_!Y}$ are given by the adjunction morphisms. Clearly, $\theta_{j_!Y}$ is a compatible inflation. Since $i^{\ast}$ is exact, we obtain a  conflation
  $$\xymatrix{i^{\ast}i_{\ast}i^!j_!Y\ar[r]^{\;\;\;i^{\ast}\theta_{j_!Y}}&i^{\ast}j_!Y\ar[r]&i^{\ast}M.}$$
  By Lemma \ref{4.2}, $i^{\ast}i_{\ast}i^!j_!Y\cong i^!j_!Y$ and $i^{\ast}j_!Y=0$. Since $i^{\ast}\theta_{j_!Y}$ is compatible and both an inflation and a deflation, $i^!j_!Y\cong 0$ for any $Y\in\mathcal{C}$. The proof is finished.
\end{proof}

Now, we will glue together Wakamatsu-tilting subcategories in $\mathcal{A}$ and $\mathcal{C}$.

\begin{theorem}\label{4.6}
   Let $(\mathcal{A},\mathcal{B},\mathcal{C})$ be a $\theta$-$\nu$-compatible recollement of extriangulated categories, $\mathcal{X}'$ and $\mathcal{X}''$ are Wakamatsu-tilting subcategories of $\mathcal{A}'$ and $\mathcal{C}$. Define
   $$\mathcal{X}=\{X\in\mathcal{B}~|~i^!X\in\mathcal{A},j^{\ast}X\in\mathcal{C}\}.$$
   If $i^!$ and $i^{\ast}$ are exact, then $\mathcal{X}$ is a Wakamatsu-tilting subcategory of $\mathcal{B}$.
\end{theorem}

\begin{proof}
  Firstly, we prove $\mathcal{X}$ is self-orthogonal. By Proposition \ref{4.3}, since $i^!$ is exact, there are two conflations
  $$i_{\ast}i^!M\stackrel{\theta_M}{\longrightarrow}M\stackrel{\vartheta_M}
  {\longrightarrow}j_{\ast}j^{\ast}M~{\rm and}~i_{\ast}i^!N\stackrel{\theta_N}{\longrightarrow}N\stackrel{\vartheta_N}
  {\longrightarrow}j_{\ast}j^{\ast}N$$ for any $M,N\in\mathcal{B}$. By applying $\mathcal{B}(-,N)$ to $i_{\ast}i^!M\stackrel{\theta_M}{\longrightarrow}M\stackrel{\vartheta_M}
  {\longrightarrow}j_{\ast}j^{\ast}M$ and $\mathcal{B}(j_{\ast}j^{\ast}M,-)$ to $i_{\ast}i^!N\stackrel{\theta_N}{\longrightarrow}N\stackrel{\vartheta_N}
  {\longrightarrow}j_{\ast}j^{\ast}N$, we obtain two long exact sequences
  $$\ldots\longrightarrow\mathbb{E}^n_{\mathcal{B}}(j_{\ast}j^{\ast}M,N) \longrightarrow \mathbb{E}^n_{\mathcal{B}}(M,N)\longrightarrow\mathbb{E}^n_{\mathcal{B}}(i_{\ast}i^!M,N)\longrightarrow\ldots\;;$$
  $$\ldots\longrightarrow\mathbb{E}^n_{\mathcal{B}}(j_{\ast}j^{\ast}M,i_{\ast}i^!N) \longrightarrow \mathbb{E}^n_{\mathcal{B}}(j_{\ast}j^{\ast}M,N)\longrightarrow\mathbb{E}^n_{\mathcal{B}}(j_{\ast}j^{\ast}M,j_{\ast}j^{\ast}N)\longrightarrow\ldots\;.$$ for any positive integer $n$.
  Since $i^!$ is exact, we obtain $$\mathbb{E}^n_{\mathcal{B}}(i_{\ast}i^!M,N)\cong\mathbb{E}^n_{\mathcal{A}}(i^!M,i^!N)=0.$$ Similarly, since $i^{\ast}$ is exact, we obtain $$\mathbb{E}^n_{\mathcal{B}}(j_{\ast}j^{\ast}M,i_{\ast}i^!N)\cong\mathbb{E}^n_{\mathcal{A}}(i^{\ast}j_{\ast}j^{\ast}M,i^!N)
  \cong\mathbb{E}^n_{\mathcal{A}}(0,i^!N)=0$$ by Lemma \ref{4.5}. By Lemma \ref{4.2}, we have $$\mathbb{E}^n_{\mathcal{B}}(j_{\ast}j^{\ast}M,j_{\ast}j^{\ast}N)\cong\mathbb{E}^n_{\mathcal{C}}(j^{\ast}j_{\ast}j^{\ast}M,j^{\ast}N) \cong \mathbb{E}^n_{\mathcal{C}}(j^{\ast}M,j^{\ast}N)=0$$ since $j_{\ast}$ is exact. Then we get $\mathbb{E}^n_{\mathcal{B}}(M,N)=0$ since $\mathbb{E}^n_{\mathcal{B}}(j_{\ast}j^{\ast}M,N)=0$.

  Next, we will show that $\mathcal{P}(\mathcal{B})\subseteq\mathcal{X}_{\mathcal{X}}$. Let $P\in\mathcal{P}(\mathcal{B})$, then $i^{\ast}P\in\mathcal{P}(\mathcal{A})$ and $j^{\ast}P\in\mathcal{P}(\mathcal{C})$. Hence, $i^{\ast}P\in\mathcal{X}_{\mathcal{X}'}$ and $j^{\ast}P\in\mathcal{X}_{\mathcal{X}''}$. We obtain two $\mathbb{E}$-triangle sequences in $\mathcal{A}$ and $\mathcal{C}$ $$\xymatrix@R=0.5ex{i^{\ast}P\ar[r]^{d_0'}&T_0'\ar[rr]^{d_1'}\ar[rd]&&T_1'\ar[rr]^{d_2'}\ar[rd]&&T_2'\ar[r]&\ldots\;;\\
  &&K_1'\ar[ru]&&K_2'\ar[ru]}
  $$$$\xymatrix@R=0.5ex{j^{\ast}P\ar[r]^{d_0''}&T_0''\ar[rr]^{d_1''}\ar[rd]&&T_1''\ar[rr]^{d_2''}\ar[rd]&&T_2''\ar[r]&\ldots\\
  &&K_1''\ar[ru]&&K_2''\ar[ru]}$$ with $T_i'\in\mathcal{X}'$, $T_i''\in\mathcal{X}''$, $K_j'\in{^{\bot}\mathcal{X}'}$ and $K_j''\in{^{\bot}\mathcal{X}''}$ for any $i\geqslant0, j\geqslant1$. Since $i_{\ast}$ and $j_!$ are exact by Lemma \ref{4.2}, we obtain two $\mathbb{E}$-triangle sequences in $\mathcal{B}$
  $$\xymatrix@R=0.5ex{i_{\ast}i^{\ast}P\ar[r]^{i_{\ast}d_0'}&i_{\ast}T_0'\ar[rr]^{i_{\ast}d_1'}\ar[rd]&&i_{\ast}T_1'\ar[rr]^{i_{\ast}d_2'}\ar[rd]&&i_{\ast}T_2'\ar[r]&\ldots\;;\\
  &&i_{\ast}K_1'\ar[ru]&&i_{\ast}K_2'\ar[ru]}
  $$$$\xymatrix@R=0.5ex{j_!j^{\ast}P\ar[r]^{j_!d_0''}&j_!T_0''\ar[rr]^{j_!d_1''}\ar[rd]&&j_!T_1''\ar[rr]^{j_!d_2''}\ar[rd]&&j_!T_2''\ar[r]&\ldots\;.\\
  &&j_!K_1''\ar[ru]&&j_!K_2''\ar[ru]}$$
  Since $i^{\ast}$ and $i^!$ is exact, we obtain that
  $$\mathbb{E}_{\mathcal{B}}(i_{\ast}i^{\ast}P,j_!T_0'')\cong\mathbb{E}_{\mathcal{A}}(i^{\ast}P,i^!j_!T_0'')\cong\mathbb{E}_{\mathcal{A}}(i^{\ast}P,0)=0\;;$$
  $$\mathbb{E}_{\mathcal{B}}(i_{\ast}K_i',j_!T_i'')\cong\mathbb{E}_{\mathcal{A}}(K_i',i^!j_!T_i'')\cong\mathbb{E}_{\mathcal{A}}(K_i',0)=0$$
  By Lemma \ref{2.12} and Lemma \ref{4.3}, there is a commutative diagram
  $$\xymatrix@C=1.5ex{j_!j^{\ast}P\ar[rr]\ar[d]_{j_!d_0''}&&P\ar[rr]\ar[d]&&i_{\ast}i^{\ast}P\ar[d]_{i_{\ast}d_0'}\\
  j_!T_0''\ar[rr]\ar[dd]_{j_!d_0''}\ar[rd]&&j_!T_0''\oplus i_{\ast}T_0'\ar[rr]\ar[dd]|!{[d]}\hole\ar[rd]&&i_{\ast}T_0'\ar[dd]_(0.3){i_{\ast}d_0'}|!{[d]}\hole\ar[rd]\\
  &j_!K_1''\ar[ld]\ar[rr]&&K_1\ar[ld]\ar[rr]&&i_{\ast}K_1'\ar[ld]\\
  j_!T_1''\ar[rr]\ar[d]_{j_!d_0''}&&j_!T_1''\oplus i_{\ast}T_1'\ar[rr]\ar[d]&&i_{\ast}T_1'\ar[d]_{i_{\ast}d_1'}\\
  \vdots&&\vdots&&\vdots}$$ where rows are conflations and columns are $\mathbb{E}$-triangle sequences. Clearly, $i^!{j_!T_n''\oplus i_{\ast}T_n'}
  \cong i^!i_{\ast}T_n'\cong T_n'\in\mathcal{X}'$ and $j^{\ast}{j_!T_n''\oplus i_{\ast}T_n'}\cong j^{\ast}j_!T_n''\cong T_n''\in\mathcal{X}''$ for any $n\geqslant0$, which induce that $j_!T_n''\oplus i_{\ast}T_n'\in\mathcal{X}$. It suffices to show that $K_n\in{^{\bot}\mathcal{X}}$ for any $n\geqslant1$. For any $N\in\mathcal{X}$, by applying ${\mathcal{B}}(-,N)$ to $j_!K_n''\longrightarrow K_n \longrightarrow i_{\ast}K_n'$, We obtain a long exact sequence in $\mathcal{B}$
  $$\ldots\longrightarrow\mathbb{E}^k_{\mathcal{B}}(i_{\ast}K_n',N) \longrightarrow \mathbb{E}^k_{\mathcal{B}}(K_n,N)\longrightarrow\mathbb{E}^k_{\mathcal{B}}(j_!K_n'',N)\longrightarrow\ldots$$
  for any $k\geqslant1$. Clearly, $\mathbb{E}^k_{\mathcal{B}}(i_{\ast}K_n',N)\cong \mathbb{E}^k_{\mathcal{A}}(K_n',i^!N)=0$ and $\mathbb{E}^k_{\mathcal{B}}(j_!K_n'',N)\cong\mathbb{E}^k_{\mathcal{C}}(K_n'',j^{\ast}N)=0$ since $i^!N\in\mathcal{X}'$ and $j^{\ast}N\in\mathcal{X}''$. Therefore, $K_n\in{^{\bot}\mathcal{X}}$ and the proof is finished.
\end{proof}

\begin{corollary}
  Let $(\mathcal{A},\mathcal{B},\mathcal{C})$ be a $\theta$-$\nu$-compatible recollement of extriangulated categories, $\mathcal{X}'$ and $\mathcal{X}''$ are Wakamatsu-cotilting subcategories of $\mathcal{A}'$ and $\mathcal{C}$. Define
  $$\mathcal{X}=\{X\in\mathcal{B}~|~i^!X\in\mathcal{A},j^{\ast}X\in\mathcal{C}\}.$$
  If $i^!$ and $i^{\ast}$ are exact, then $\mathcal{X}$ is a Wakamatsu-cotilting subcategory of $\mathcal{B}$.
\end{corollary}

\begin{proof}
  The proof is clear by Lemma \ref{3.6} and Lemma \ref{4.6}.
\end{proof}

Now since we have glued the Wakamatsu-tilting subcategories, it is natural to consider whether the converse of the result is true. The following theorem shows that the converse of the gluing holds under certain conditions.

\begin{theorem}\label{4.8}
  Let $(\mathcal{A},\mathcal{B},\mathcal{C})$ be a recollement of extriangulated categories and $\mathcal{X}$ be a Wakamatsu-tilting subcategory of $\mathcal{B}$. If $i^!$ is exact, $i_{\ast}i^!({^{\bot}\mathcal{X}})\subseteq{^{\bot}\mathcal{X}}$ and $j_{\ast}j^{\ast}(\mathcal{X}^{\bot})\subseteq\mathcal{X}^{\bot}$, then $i^!\mathcal{X}$ and $j^{\ast}\mathcal{X}$ are Wakamatsu-tilting subcategories of $\mathcal{A}$ and $\mathcal{C}$, respectively.
\end{theorem}

\begin{proof}
  Firstly, we show that $i^!\mathcal{X}$ and $j^{\ast}\mathcal{X}$ are self-orthogonal. For any $M,N\in\mathcal{X}$, since $i^!$ and $j_{\ast}$ are exact, we have $\mathbb{E}^n_{\mathcal{A}}(i^!M,i^!N)\cong \mathbb{E}^n_{\mathcal{B}}(i_{\ast}i^!M,N)=0$ and
  $\mathbb{E}^n_{\mathcal{C}}(j^{\ast}M,j^{\ast}N)\cong\mathbb{E}^n_{\mathcal{B}}(M,j_{\ast}j^{\ast}N)=0$.

  It suffices to show that $\mathcal{P}(\mathcal{A})\subseteq\mathcal{X}_{i^!\mathcal{X}}$ and $\mathcal{P}(\mathcal{C})\subseteq\mathcal{X}_{j^{\ast}\mathcal{X}}$. For any $P\in\mathcal{P}(\mathcal{A})$ and $Q\in\mathcal{P}(\mathcal{C})$, $i_{\ast}P, j_!Q\in \mathcal{P}(\mathcal{B})$ by Lemma \ref{4.2}. We obtain two $\mathbb{E}$-triangle sequences in $\mathcal{B}$
  \begin{equation}\label{les1}
  \xymatrix@R=0.5ex{i_{\ast}P\ar[r]^{d_0'}&T_0'\ar[rr]^{d_1'}\ar[rd]&&T_1'\ar[rr]^{d_2'}\ar[rd]&&T_2'\ar[r]&\ldots\;;\\
  &&K_1'\ar[ru]&&K_2'\ar[ru]}
  \end{equation}
  \begin{equation}\label{les2}
  \xymatrix@R=0.5ex{j_!Q\ar[r]^{d_0''}&T_0''\ar[rr]^{d_1''}\ar[rd]&&T_1''\ar[rr]^{d_2''}\ar[rd]&&T_2''\ar[r]&\ldots\\
  &&K_1''\ar[ru]&&K_2''\ar[ru]}
  \end{equation} with $T_i', T_i''\in\mathcal{X}$ and $K_j', K_j''\in{^{\bot}\mathcal{X}''}$ for any $i\geqslant0, j\geqslant1$. By applying $i^!$ to (\ref{les1}) and $j^{\ast}$ to (\ref{les2}), we obtain two $\mathbb{E}$-triangle sequences in $\mathcal{A}$ and $\mathcal{C}$
  $$\xymatrix@R=0.5ex{i^!i_{\ast}P\ar[r]^{i^!d_0'}&i^!T_0'\ar[rr]^{i^!d_1'}\ar[rd]&&i^!T_1'\ar[rr]^{i^!d_2'}\ar[rd]&&i^!T_2'\ar[r]&\ldots\;;\\
  &&i^!K_1'\ar[ru]&&i^!K_2'\ar[ru]}$$
  $$\xymatrix@R=0.5ex{j^{\ast}j_!Q\ar[r]^{j^{\ast}d_0''}&j^{\ast}T_0''\ar[rr]^{j^{\ast}d_1''}\ar[rd]&&j^{\ast}T_1''\ar[rr]^{j^{\ast}d_2''}\ar[rd]&&j^{\ast}T_2''\ar[r]&\ldots\;.\\
  &&j^{\ast}K_1''\ar[ru]&&j^{\ast}K_2''\ar[ru]}$$ Clearly, $i^!T_i'\in i^!\mathcal{X}$ and $j^{\ast}T_i''\in j^{\ast}\mathcal{X}$ for any $i\geqslant0$. Since $i^!$ and $j_{\ast}$ are exact, we obtain that $\mathbb{E}^n_{\mathcal{A}}(i^!K_j',i^!N)\cong \mathbb{E}^n_{\mathcal{B}}(i_{\ast}i^!K_j',N)=0$ and
  $\mathbb{E}^n_{\mathcal{C}}(j^{\ast}K_j'',j^{\ast}N)\cong\mathbb{E}^n_{\mathcal{B}}(K_j'',j_{\ast}j^{\ast}N)=0$ for any $j\geqslant1$ and $N\in\mathcal{B}$. Therefore, the proof is finished.
\end{proof}

\begin{corollary}
  Let $(\mathcal{A},\mathcal{B},\mathcal{C})$ be a recollement of extriangulated categories and $\mathcal{X}$ be a Wakamatsu-cotilting subcategory of $\mathcal{B}$. If $i^!$ is exact, $i_{\ast}i^!({^{\bot}\mathcal{X}})\subseteq{^{\bot}\mathcal{X}}$ and $j_{\ast}j^{\ast}(\mathcal{X}^{\bot})\subseteq\mathcal{X}^{\bot}$, then $i^!\mathcal{X}$ and $j^{\ast}\mathcal{X}$ are Wakamatsu-cotilting subcategories of $\mathcal{A}$ and $\mathcal{C}$, respectively.
\end{corollary}

\begin{proof}
  The proof is clear by Lemma \ref{3.6} and Theorem \ref{4.8}.
\end{proof}


\end{document}